\documentclass[12pt, ?4]{article}
\usepackage{amssymb, amsfonts, amsmath}
\usepackage{amsthm}
\usepackage[english]{babel}
\usepackage[mathcal]{eucal}
\usepackage{mathrsfs}

\newtheorem{theorem}{Theorem}
\newtheorem{lemma}{Lemma}[section]

\newtheorem{corollary}{Corollary}

\begin{document}
	
	\title{Central extensions of free periodic groups}
	\author{S.\,I.~Adian, V.\,S.~Atabekyan}
	
	
\date{}
	
	\maketitle

		\begin{abstract}
			
			It is proved that any countable abelian group $D$ can be embedded as a centre into a $m$-generated group $A$ such that the quotient group $A/D$ is isomorphic to the free Burnside group $B(m,n)$ of rank $m>1$ and of odd period $n\ge665$. The proof is based on some modification of the method which was used by S.I.Adian in his monograph in 1975 for a positive solution of Kontorovich's famous problem from Kourovka Notebook on the existence of a finitely generated non-commutative analogue of the additive group of rational numbers.  More precisely, he proved that the desired analogues in which the intersection of any two non-trivial subgroups is infinite, can be constructed as a central extension of the free Burnside group $ B (m, n) $, where $ m> 1 $, and $ n \ge665 $ is an odd number, using as its center the infinite cyclic group. The paper also discusses other applications of the proposed generalization of Adian's technique. In particular, we describe the free groups of the variety defined by the identity $[x^n,y]=1$ and the Schur multipliers of the free Burnside groups $B(m,n)$ for any odd $n\ge665$.
			
		\end{abstract}
		

		\markright{Central extensions of free periodic groups}
		
\footnotetext[0]{The first author was supported by a grant from the Russian Science Foundation (project no. 14-50-00005) at the Steklov Mathematical Institute of the Russian Academy of Sciences. The second author was supported by the Republic of Armenia MES State Committee of Science and Russian Foundation for Basic Research (RF) in the frames of the joint research project SCS 18RF-109 and RFBR 18-51-05006 accordingly.}

		\section{Introduction}
		
	In this paper we study the central extensions of infinite periodic groups of a bounded odd exponent. A \textit{central extension} of a group $B$ by a group $C$ is a group $A$ the centre of which is isomorphic to $C$, and the quotient group $A/C$ is isomorphic to $B$. It is clear that if a group $A$ is a central extension of $B$ by $C$, then $C$ must be an abelian group. In the studies of central extensions of periodic groups the method of inductive proof of large series of interdependent assertions about periodic words is used, allowing one to construct examples of new groups with the given properties. Originally, this method was created in 1968 by S.I. Adian and P.S. Novikov to solve the Burnside problem on the existence of infinite periodic groups of a sufficiently large odd period and with a finite number of generators.
		
	Further development of this approach with applications of its various modifications in solutions of some other problems of group theory was published in the monograph of S.I. Adian \cite{A} (see also \cite{A15}). In particular, S.I.Adian constructed in \cite {A71} non-abelian analogues of the group of rational numbers, that is, non-abelian groups in which the intersection of any two non-trivial subgroups is infinite (a positive solution of P.G.Kontorovich's problem, see \cite{KT}, Question 1.63)). The central extensions constructed in The central extensions constructed there can be regarded as analogues of the classical rational numbers group $ \mathbb{Q} $, since among the abelian groups only the group $ \mathbb{Q} $ and its subgroups have the indicated infiniteness property of the intersection of any two nontrivial subgroups.  can be regarded as analogues of the classical rational numbers group $ \mathbb{Q} $, since among the abelian groups only the group $ \mathbb{Q} $ and its subgroups have the indicated infiniteness property of the intersection of any two nontrivial subgroups. This analogues were denoted by $A(m,n)$. Any group $A(m,n)$ is a central extension of the free Burnside group $B(m,n)$ of rank $m$ and of a fixed odd period $n\ge665$ by an infinite cyclic group.  By definition the free Burnside group $B(m,n)$ of period $n$ and rank $m$ has the following presentation:  $$B(m,n)=\langle a_1, a_2, ..., a_m \mid X^n=1\rangle, $$
	where $X$  runs through the set of all words in the alphabet $\{a_1^{\pm1},a_2^{\pm1},\ldots,a_m^{\pm1}\}$. The group $B(m,n)$ is the quotient group of the free group $F_m$ of rank $m$ by the normal subgroup $F_{m}^n$ generated by all possible $n$-th powers of the elements of $F_m$. The presentation of the group $A(m,n)$ by generators and defining relations is easily obtained from the definition of the group $B(m,n)$ on the basis of the independent system of defining relations $\{A^n=1\mid A\in\mathcal{E}\}$ constructed in Ch. VI of Adian's monograph \cite{A} as a result of the addition to the set of generators $\mathcal{A}=\{a_1, a_2,\ldots, a_m\}$ a new letter $d$, which commutes with all $a_i$, and the replacement of each relation $A^n = 1$ by $A^n=d$. Later it was established that if one adds another relation $d^n=1$ to the defining relations of the group $A(m,n)$, then the resulting periodic quotient group $A'(m,n)$ only admits the discrete topology (see \cite{A80}, item 13.4). Thus, an answer  was given to A.A.Markov's question about the existence of an untopologizable countable group that remained open for several decades (see also \cite{O80}). 
	Among the various applications of the groups $A(m,n)$ we  also note the recent work \cite{MSL} where the group $A(m,n)$ is used to give a description of $\{2,3\}$-groups acting freely on some non-trivial abelian group.

	We propose some generalization of Adian's method in order to construct central extensions of periodic groups $B(m,n)$ using arbitrary abelian groups, including groups that are not finitely generated. Some applications of this method will also be indicated. Note that this generalization can also be applied to $ n $-periodic products, as well as to the central extensions of free groups of S.I.Adian's infinitely based varieties (see \cite{AA17c}, \cite{AA17i}).
	
	\medskip We proceed to precise definitions. The main ideas of our constructions are borrowed from the paragraph 1 of Chapter VII of the monograph \cite{A}  where the groups $A(m,n)$ are constructed. Section 3 of this paper differs little from Section 1 of Chapter VII of the monograph \cite{A} and, in fact, could be written back in 1975.
	
	We fix an integer $ m>1 $ and an odd integer $n\ge 665$. Consider the set of elementary words $\mathcal{E}$ which is defined in \cite[VI.2.1]{A}. Recall the definition of this set. For every $\alpha>0$ we choose the set $\mathcal{E}_\alpha$ consisting of elementary periods $A$ of rank $\alpha$ in the group alphabet $\{a_1, a_2,..., a_m\}$ such that the following conditions are satisfied:
	
	(a) for every elementary word $E$ of rank $\alpha$ there exists one and only one word $A$ such that $$A\in \mathcal{E}_\alpha,\,\,\,\text{and} \,\,\,(\,\text{Rel}(E,A^n) \,\,\,\text{or}\,\,\, \text{Rel}(E,A^{(-n)})),$$
	
	(b) if $A\in \mathcal{E}_\alpha$, then for some $P$ and $Q$ we have $PA^nQ\in \overline{\mathcal{M}}_{\alpha-1}$.

	Let $ \mathcal{E} $ denote the set $ \bigcup^\infty_{\alpha=1}\mathcal{E}_\alpha $: \begin{equation}\label{E}\mathcal{E}=\bigcup^\infty_{\alpha=1}\mathcal{E}_\alpha. \end{equation}

	The set $\mathcal{E}$ is countable (see Theorem 2.13 of Chapter VI \cite{A}), that is, its elements can be enumerated by natural numbers. We fix some numbering and let $\mathcal{E}=\{A_j| j\in \mathbb{N}\}$ ($\mathbb{N}$ is the set of natural numbers).

	We also fix an at most countable arbitrary abelian group $\mathcal{D}$, given by generators and defining relations in the following form:
	\begin{equation}\label{D} \mathcal{D}=\langle d_1, d_2,\ldots, d_i,\ldots\,\,|\,\,r=1,\,r\in \mathcal{R}\rangle, \end{equation}
	where $\mathcal{R}$ is some set of words in the group alphabet $d_1, d_2,\ldots, d_i,\ldots$

	By $A_\mathcal{D}(m,n)$ we denote the group given by a system of generators of two types:
	
	\begin{equation}\label{a}a_1, a_2,..., a_m\end{equation}
	and
	\begin{equation}\label{d} d_1, d_2, ..., d_i, ....\end{equation}
	and by a system of defining relations of the following three types:
	
	\begin{equation}\label{k0} r=1,\,\,\,\text{for all}\,\,\,r\in \mathcal{R},\end{equation}
	
	\begin{equation}\label{k1} a_id_j=d_ja_i,\end{equation}
	
	\begin{equation}\label{k2}  A_j^n=d_j\end{equation}
	for all $A_j\in \mathcal{E}$, $i=1,2,...,m$ and $j\in \mathbb{N}$.
	
	Note that if as a group $ \mathcal{D} $ one takes an infinite cyclic group $$\mathcal{D}_0=\langle d_1, d_2,\ldots, d_i,\ldots\,\,|\,d_jd_k^{-1}=1,\,j,k\in \mathbb{N}\rangle, $$ then the obtained group $A_{\mathcal{D}_0}(m,n)$ will exactly coincide with the group $A(m,n)$ of the monograph \cite{A}.
	
	From the defining relations \eqref{k2} it follows that the groups $A_\mathcal{D}(m,n)$ are $m$-generated groups with the generators \eqref{a}. For the groups $A_\mathcal{D}(m,n)$ the following main theorem holds.
	
	\begin{theorem}\label{t1} For every $m>1$ and odd $n\ge665$ and for any abelian group $\mathcal{D}$ having the presentation \eqref{D} the following conditions hold:
		
		1. The identity $[x^n,y]=1$ holds in the group $A_\mathcal{D}(m,n)$,
		
		2. The verbal subgroup of $A_\mathcal{D}(m,n)$ generated by the word $x^n$ coincides with $\mathcal{D}$,
		
		3. The centre of $A_\mathcal{D}(m,n)$ coincides with $\mathcal{D}$,
		
		4. The quotient group of $A_\mathcal{D}(m,n)$ by the subgroup $\mathcal{D}$ is the free Burnside group $B(m,n)$. \end{theorem}

	\textbf{Remark 1.} Above we only built the groups $A_{\mathcal{D}}(m,n)$ for at most countable abelian groups. By increasing the rank $m$ to a given cardinality an abelian group of arbitrary cardinality can be verbally embedded into the corresponding group $A_{\mathcal{D}}(m,n)$ in the similar way.
	
	\textbf{Remark 2.} 
Let us pay attention to a certain freedom in constructing groups $A_{\mathcal{D}}(m,n)$. First, there is a certain arbitrariness in the order of the numbering of elementary periods $A_j\in \mathcal{E}$, $j\in \mathbb{N}$. Second, we have a large degree of freedom when choosing the presentation \eqref{D} of an abelian group $\mathcal{D}$. Thus, by virtue of the defining relations of the form \eqref{k2}, for a fixed group $\mathcal{D}$ we can get different groups $A_{\mathcal{D}}(m,n)$. Moreover, for each of them Theorem \ref{t1} holds.
		
	The next question was posed by P. de la Harpe in Kourovka Notebook \cite{KT} (see Question 14.10 b)): \textit{Find an explicit embedding of the additive group of rational numbers $\mathbb{Q}$ into a finitely generated group (such an embedding exists by Theorem IV from the paper G.Higman, B.H.Neumann, H.Neumann, J. London Math. Soc., 24 (1949), 247--254)}.
	
	As a group \eqref{D} we choose the additive group of rational numbers $ \mathbb{Q} $ given, for example, by the representation:
	
	\begin{equation}\label{q}\mathbb{Q}=\langle d_1, d_2,\ldots, d_i,\ldots\,\,|\,d_i^i=d_{i-1}, i\ge 2\rangle.\end{equation}

	Let us construct the group $A_\mathbb{Q}(m,n)$ by generators \eqref{a}, \eqref{d} and by defining relations \eqref{k0}--\eqref{k2} where $\mathcal{R}=\{d_i^{-i}d_{i-1}, i\ge 2\}$. Then we get a verbal embedding of the group $\mathbb{Q}$ into an $m$-generated group $A_\mathbb{Q}(m,n)$ for any $ m>1 $, having an explicit presentation of the form  \eqref{k0}--\eqref{k2}. Moreover, the group $\mathbb{Q}$ coincides with the centre of $A_\mathbb{Q}(m,n)$. We note that another explicit (subnormal) embedding of the group $ \mathbb{Q} $ into some finitely generated group was previously proposed in \cite{VM}. An effective embedding of the group $ \mathbb{Q} $ into a finitely generated group can also be found in the paper \cite{G}.\footnote {The authors are grateful to the referee who drew their attention to the work \cite{G}.}
	
	The following assertion, in particular, describes all finite subgroups of $A_\mathbb{Q}(m,n)$.
	
	\begin{corollary}\label{c0}If the group $\mathcal{D}$ is torsion-free, then any finite subgroup of the group $A_{\mathcal{D}}(m,n)$ is contined in some cyclic subgroup of order $n$. 
	\end{corollary}
	\begin{proof}  Let $H$ be a finite subgroup of the group $A_\mathcal{D}(m,n)$. Since the center $\mathcal{D}$ of  $A_\mathcal{D}(m,n)$ is torsion-free (see item 3 of Theorem \ref{t1}), then $H\cap \mathcal{D}=\{1\}$. Hence, the image $\overline{H}\simeq H\mathcal{D}/\mathcal{D}$ of subgroup $H$ under natural epimorphism of $A_\mathcal{D}(m,n)$ to the quotient group $A_\mathcal{D}(m,n)/\mathcal{D}\simeq B(m,n)$ is isomorphic to $H/H\cap\mathcal{D}\simeq H$. By Theorem 1.8 of Chapter VII \cite{A} all finite subgroups of the group $B(m,n)$ are cyclic. In particular, $\overline{H}$ is a cyclic group of order dividing $n$. Therefore, $H$ is also a cyclic group of the same order.
	\end{proof}

	Note that $ A_\mathbb{Q}(m, n) $ contains elements of order $ n $. Indeed, since the group $\mathbb{Q}$ is a complete group, then for each $ j\ge1 $ there is an element $q_j\in\mathbb{Q}$ such that $q^n_j=d_j $, where $ d_j $ satisfies the relation \eqref{k2}. Elementary periods $ A_j $ have order $ n $ in the group $ B(m, n) \simeq A_\mathbb{Q}(m, n) / \mathbb{Q}$. Therefore by the relations \eqref{k2} the elements of the form $ A_jq_j^{-l} $ also have order $ n $ in $ A_\mathbb{Q}(m, n) $. Thus, the order of any element of the $ m $ -generated group $ A_ \mathbb{Q} (m, n) $ is either infinite or divides $ n $.
	
	Let us consider one more interesting example. As an abelian group $ \mathcal{D} $ we take the free abelian group of countable rank
	\begin{equation}\label{C} \mathcal{C}=\langle d_1, d_2,\ldots, d_i,\ldots\,\,|\,\,[d_k,d_j]=1,\,k,j\in\mathbb{N}\rangle \end{equation}
	with the free generators $d_1, d_2,\ldots, d_i,\ldots$. The group $ A_{\mathcal{D}}(m,n)$ constructed by $ \mathcal{C} $ is denoted by $ A_{\mathcal{C}}(m,n) $. For the group $ A_{\mathcal{C}}(m,n) $ the following theorems are true.
	
	\begin{theorem}\label{t3}For any $m>1$ and odd $n\ge665$ the group $A_{\mathcal{C}}(m,n)$ is the free group of rank $m$ in the variety of groups $\mathfrak{C}$ defined by the identity $[x^n,y]=1$. \end{theorem}
	
	Hence, $A_{\mathcal{C}}(m,n)=F_m/[F_m,N]$, where $F_m$ is the absolutely free group of rank $m$ and $N=F_m^n$ is the verbal subgroup of $ F_m $ generated by the word $x^n$.
	
	\begin{theorem}\label{t4} The group $A(m,n)$ is a quotient group of $A_{\mathcal{C}}(m,n)$ by some central subgroup. \end{theorem}
	
	Applying item 3 of Theorem \ref{t1} to the group $A_{\mathcal{C}}(m,n)$ we get
	
	\begin{corollary}\label{c1} The centre of $A_{\mathcal{C}}(m,n)$ is a free abelian group of countable rank. Moreover, the elements $\{d_j|j\in \mathbb{N}\}$ are free generators for the centre of $A_{\mathcal{C}}(m,n)$. \end{corollary}
	
	Previously I.S.Ashmanov and A.Yu.Olshanskii proved in \cite{AO85} the assertion of Corollary \ref{c1} for all odd periods $n>10^{10}$ (see also \cite{O89}, Theorem 31.2).
		
	\begin{corollary}\label{c2} The Schur multiplier of the free Burnside group $B(m,n)$ is a free abelian group of countable rank for all $m>1$ and odd periods $n\ge665$. \end{corollary}
	\begin{proof} By definition the Schur multiplier $M$ of the group $B(m,n)$ is the quotient group $M=N\cap[F_m,F_m]/[F_m,N]$, where $F_m$  is the free group of rank $m$, and $N=F_m^n$. Obviously, $M<N/[F_m,N]$. By virtue of items 2 and 3 of Theorem \ref{t1} the group $N/[F_m,N]$ coincides with the center of $A_{\mathcal{C}}(m,n)$. According to Corollary \ref{c1} the group $N/[F_m,N]$ is an abelian group of countable rank.
		
		The quotient group of $N/[F_m,N]$ by the subgroup $M$ is isomorphic to $$N/N\cap[F_m,F_m]\simeq N[F_m,F_m]/[F_m,F_m],$$ which is a subgroup of the free abelian group $F_m/[F_m,F_m]$ of rank $m$. Hence, $M$ also is a free abelian group of countable rank. \end{proof}
	
	Corollary \ref{c2} strengthens Corollary 1 of the paper \cite{AO85}, although their proofs essentially coincide. We repeated this proof in view of its brevity.
	
	From Theorem \ref{t4} and Corollary \ref{c1} also follows
	
	\begin{corollary}\label{c3} The group $A_{\mathcal{C}}(m,n)$ is torsion-free. \end{corollary}
	
	\begin{proof} By virtue of Theorem \ref {t4} the quotient group of $A_{\mathcal{C}}(m,n)$ by some central subgroup is isomorphic to $A(m,n)$. By Theorem 1.6 of Chapter VII of the monograph \cite{A} the group $A(m,n)$ is torsion-free. Hence, any element of finite order of $A_{\mathcal{C}}(m,n)$ belongs to the center. But the center of $A_{\mathcal{C}}(m,n)$ is also torsion-free by Corollary \ref{c1}. \end{proof}
	
	Theorems \ref{t3} and \ref{t4} follow from the definitions of the groups $A_{\mathcal{C}}(m,n)$, $A(m,n)$ and of the variety $\mathfrak{C}$. They will be proved in Section 2. For the proof of Theorem \ref{t1} we need some generalization of the method of studying groups $A(m,n)$ of the monograph of S.I.Adian \cite{A}. Sections 3 and 4 are devoted to the mentioned generalization. Theorem \ref{t1} we will prove in Section 5.
	
	\section{The proof of Theorems \ref{t3} and \ref{t4}}
	
	\begin{lemma}\label{l1} The subgroup $\mathcal{Z}$ generated by the elements $\{d_j|j\in \mathbb{N}\}$ coincides with the centre of the group $A_{\mathcal{D}}(m,n)$, and the quotient group
		$A_{\mathcal{D}}(m,n)/ \mathcal{Z}$ is isomorphic to the free periodic group $B(m,n)$. \end{lemma}
	
	\begin{proof} By virtue of the defining relations \eqref{k1} and \eqref{k2} and by Theorem \cite[VI.2.9]{A}, $\mathcal{Z}$ is contained in the centre of $A_{\mathcal{D}}(m,n)$. Besides, the quotient group of the group $A_{\mathcal{D}}(m,n)$ by the subgroup $\mathcal{Z}$ is the free Burnside group $B(m,n)$. Moreover, according to Theorem \cite[VI.3.4]{A} the quotient group $A_{\mathcal{D}}(m,n)/ \mathcal{Z}$ has trivial centre. Hence, $\mathcal{Z}$ coincides with the centre of $A_{\mathcal{D}}(m,n)$. \end{proof}
	
	\begin{lemma}\label{l2} The verbal subgroup of $A_\mathcal{D}(m,n)$ generated by the word  $x^n$ coincides with $\mathcal{Z}$, and in the group $A_\mathcal{D}(m,n)$ the identity relation $[x^n,y]=1$ holds. \end{lemma}
	
	\begin{proof} It follows from Lemma \ref{l1} that every element of $A_{\mathcal{D}}(m,n)$ can be presented in the form $Xd$, where $X$ is a word in the alphabet \eqref{a}, $X^n\in \mathcal{Z}$ and $d\in \mathcal{Z}$. For any $Xd\in A_{\mathcal{D}}(m,n)$ the element $(Xd)^n$ belongs to the centre of $A_{\mathcal{D}}(m,n)$ because the quotient group $A_{\mathcal{D}}(m,n)/\mathcal{Z}$ has the period $n$. Hence, in $A_{\mathcal{D}}(m,n)$ the identity relation $[x^n,y]=1$ holds. Besides, it follows from the defining relations \eqref{k2} that the verbal subgroup of the group $A_\mathcal{D}(m,n)$ generated by the word $x^n$ coincides with $\mathcal{Z}$ since the centre $ \mathcal{Z} $ is generated by the elements $\{d_j|j\in \mathbb{N}\}$. \end{proof}
	
 It is easy to see that excluding from the system of generators of the group $ A_{\mathcal{C}}(m,n) $ all the letters $d_j$, $j\in \mathbb{N}$, we get the following presentation for the groups $A_{\mathcal{C}}(m,n)$: $$A_{\mathcal{C}}(m,n)=\langle a_1,a_2,\ldots, a_m\,|\, [A_j^n,a_k]=1,\,j\in \mathbb{N},\,k=1,2,\ldots, m\rangle.$$
	This implies that every relation of the group $ A_{\mathcal{C}}(m,n) $ is a consequence of the identity relation $[x^n,y]=1$. On the other hand by virtue of Lemma \ref{l2} the identity $[x^n,y]=1$ holds in $A_{\mathcal{C}}(m,n)$. Thus, the group $A_{\mathcal{C}}(m,n)$ is the free group of rank $m$ of the variety $\mathfrak{C}$ defined by the identity $[x^n,y]=1$. Theorem \ref{t3} is proved.
	
	\bigskip
	Directly from the definitions of the groups $ A(m,n) $ and $ A_{\mathcal{C}}(m,n) $ it implies that the group $ A(m,n) $ is the quotient subgroup of $ A_{\mathcal{C}}(m,n) $ by the subgroup $ H $ generated by all possible elements of the form
	$d_id_j^{-1}$, $ i,j\in \mathbb{N}.$ Since this elements $d_id_j^{-1}$ belong to the centre of $A_{\mathcal{C}}(m,n)$, then the group $A(m,n)$ is the quotient group of $A_{\mathcal{C}}(m,n)$ by some central subgroup. Theorem \ref{t4} is proved.
	
	\section{Generalized analogues of the basic concepts of Adian-Novikov theory}
	For the words in the alphabet \eqref{a}--\eqref{d} of the group $A_\mathcal{D}(m,n)$ we construct generalized analogues of concepts that were constructed and studied in chapters I-V of the monograph \cite{A}.
	
	The set of all words of the form $Qd$ is denoted by $\mathcal{R}_\alpha^\mathcal{D}$, where $Q$ is a word in the alphabet \eqref{a} belonging to the set $\mathcal{R}_\alpha$, $d\in \mathcal{D}$ and the group $\mathcal{D}$ has the presentation \eqref{D}. Similarly, denote by $\mathcal{N}_\alpha^\mathcal{D}$, $\mathcal{P}_\alpha^\mathcal{D}$, $\mathcal{K}_\alpha^\mathcal{D}$, $\mathcal{L}_\alpha^\mathcal{D}$, $\mathcal{M}_\alpha^\mathcal{D}$, $\overline{\mathcal{M}}_\alpha^\mathcal{D}$ and $\mathcal{A}_\alpha^\mathcal{D}$ the sets of words $Qd$, where $d\in \mathcal{D}$, and $Q$ belongs to the corresponding set $\mathcal{N}_\alpha$, $\mathcal{P}_\alpha$, $\mathcal{K}_\alpha$, $\mathcal{L}_\alpha$, $\mathcal{M}_\alpha$, $\overline{\mathcal{M}}_\alpha$ and $\mathcal{A}_\alpha$.
	
	We will only consider those occurrences in the words $Qd$, whose bases occur in $Q$, that is, they are words in the alphabet \eqref{a}. Furthermore, if $ V $ is an occurrence in the word $ Q $, then $ Vd $ denotes the corresponding occurrence in $Qd$.
	
	The concepts of periodic, integral, semi-integral and elementary words of rank $\alpha$, of generating occurrence of rank $\alpha$ and of supporting kernel of rank $\alpha$ and all concepts, which were defined in \cite[Ch. I, items 4.3-4.10]{A}, remain unaltered. They will, as before, relate only to words in the alphabet  \eqref{a}. We will study the occurrences of elementary words of rank $\alpha$ in the words of the form $Qd$, where $Q\in \mathcal{R}_{\alpha-1}$. All the concepts defined in \cite[Ch. I, items 4.11-4.15]{A} extend uniquely to such occurrences. In particular, we have in addition that for any 
	$d, d'\in \mathcal{D}$  
	$$
	\text{Corr}(V,W)\quad\Leftrightarrow\quad\text{Corr}(Vd,Wd'),
	$$ $$
	\text{Rel}(V,W)\quad\Leftrightarrow\quad\text{Rel}(Vd,Wd'),
	$$ 
	$$
	V\in\text{Norm}(\alpha,Q,r)\quad\Leftrightarrow\quad V\in\text{Norm}(\alpha,Qd,r).
	$$
	
	We extend the concepts of kernel of rank $\alpha$ and the relation $\text{MutNorm}_\alpha(V,W)$ to occurrences in words of the form $Qd$, so that the relations 
	$$
	V\in\text{Ker}(\alpha,Q)\quad\Leftrightarrow\quad V\in\text{Ker}(\alpha,Qd),$$$$V\in\text{Reg}(\alpha,Q)\quad\Leftrightarrow\quad V\in\text{Reg}(\alpha,Qd),$$$$\text{MutNorm}_\alpha(V,W)\quad\Leftrightarrow\quad\text{MutNorm}_\alpha(Vd,Wd)
	$$ hold.
    	
	Further, by simultaneous induction on rank $\alpha$ we will define on $\mathcal{R}^\mathcal{D}_\alpha$ a relation $\mathop\eqsim\limits^\alpha$ of generalized equivalence in rank $\alpha$ and an operation $[X,Y]^\mathcal{D}_\alpha$ of generalized coupling in rank $\alpha$.
	
	We will parallelly prove that the relation $\mathop\eqsim\limits^\alpha$ satisfies the following conditions:
	\begin{equation}\label{5} P\mathop\eqsim\limits^\alpha Q\Leftrightarrow\exists d\,\forall d'\, (Pd' \mathop\eqsim\limits^\alpha Q(dd')), \end{equation}
	\begin{equation}\label{6} Qd\mathop\eqsim\limits^\alpha Qd'\Rightarrow d=d' \,\,\text{â}\,\, \mathcal{D},
	\end{equation}
	where $P,Q$ are words in the alphabet \eqref{a}, $d,d'\in \mathcal{D}$, $(dd')$ is the product of $d$ and $d'$ in abelian group $\mathcal{D}$.
	
	Essentially, in order to construct the group $ A_\mathcal{D}(m,n)$ mentioned in Theorem \ref{t1} we consider the direct product of the free group $ F_m $ with free generators $ a_1, ..., a_m $ and the abelian group $\mathcal{D}$ \eqref{D}. The elements of this product we denote by $Qd$, where $Q\in F_m$ and $d\in \mathcal{D}$. Further, on this direct product we introduce the equivalence relation $\eqsim$ ($\mathop\eqsim\limits^\alpha $) by induction on the rank based on the following considerations. If $Q, P\in F_m$ and $Q=P$ in the free Burnside group $B(m,n)$, then for some words $u_1,..., u_k$ and elementary periods $A_{i_1}, ... ,A_{i_k}\in \mathcal{E}$ in free group $F_m$ we have the equality
	$$Q=P(u_1A_{i_1}^{\pm n}u_1^{-1})\cdots (u_kA_{i_k}^{\pm n}u_k^{-1}).$$
	Taking into account the relations \eqref{k2}, \eqref{k1} we for some $d'\in \mathcal{D}$ get the equality $Q=Pd'$  in the group $A_\mathcal{D}(m,n)$. Then we assume that $Qd\eqsim P(d'd)$. To prove that after this factorization we obtain a group isomorphic to the group $A_\mathcal{D}(m,n)$ and that the group $ \mathcal{D} $ is embedded in it, in particular, we have to prove the relations \eqref{5} and \eqref{6} together with the concomitant statements.  
	
	Looking ahead, we note that the correctness of the definition of the generalized equivalence relation $ \mathop\eqsim\limits^\alpha $ (more precisely, its independence from the sequence of generalized rotations of rank $ \alpha $) is ensured by the \textit{commutativity} of the group $\mathcal{D}$.
	
	For arbitrary $P,Q\in \mathcal{R}_0$ we put by definition  $$Pd\mathop\eqsim\limits^0 Qd'\Leftrightarrow (P\mathop\sim\limits^0 Q\,\,\, \text{and}\,\,\, d=d') \Leftrightarrow Pd\equiv Qd',$$
	and \begin{equation}\label{7}[Pd,Qd']^\mathcal{D}_0\rightleftharpoons[P,Q]_0(dd').
	\end{equation}
	Clearly, relations \eqref{5} and \eqref{6} hold if $\alpha=0$.
	
	Suppose that $\alpha>0$ and we have already defined the relation $\mathop\eqsim\limits^{\alpha-1}$ on $\mathcal{R}^\mathcal{D}_{\alpha-1}$ satisfying \eqref{5} and \eqref{6} and the operation $[X,Y]^D_{\alpha-1}$ satisfying
	\begin{equation}\label{8} [Pd,Qd']^\mathcal{D}_{\alpha-1}\mathop\eqsim\limits^{\alpha-1}[P,Q]^\mathcal{D}_{\alpha-1}(dd') \end{equation}
	and
	\begin{equation}\label{16} [Pd,Qd']^\mathcal{D}_{\alpha-1}=[P_1d'',Q_1d''']^\mathcal{D}_{0}=[P_1,Q_1]_0(d''d''') \end{equation}
	provided that the relations
	\begin{equation}\label{17} Pd\mathop\eqsim\limits^{\alpha-1}P_1d'', \,\, Qd'\mathop\eqsim\limits^{\alpha-1}Q_1d'''\,\,\,\text{and}\,\,\,[P_1,Q_1]_0\in\mathcal{R}_{\alpha-1} \end{equation} hold.
	
	The existence of such an operation on $\mathcal{R}^\mathcal{D}_{\alpha-1}$, the fact that it is associative and well defined within generalized equivalence in rank $\alpha-1$ are given us by the induction assumption. For $\alpha=1$, these properties of the operation $[X,Y]^\mathcal{D}_{\alpha-1}$ follow trivialy from \eqref{7}.
	
	In order to define $\mathop\eqsim\limits^\alpha$ we now need the concept of a generalized reversal of rank $\alpha$ which is connected with the usual concept of reversal of rank $\alpha$.
	
	Suppose that $X\in \mathcal{R}_{\alpha-1}$ and the transition $X\to X_1$ is a real reversal of rank $\alpha$ of the occurrence  $R\ast E\ast S\in\text{Norm}(\alpha,X,9)$. By \cite[Ch. VI, Lemma 2.1]{A} there exists an element $A^n_j\in\mathcal{E}_\alpha$ such that either $\text{Rel}(E, A^n_j)$ or $\text{Rel}(E, A^{-n}_j)$. As it was shown in \cite[Ch. VI, Lemma 2.3]{A} in this case we can find a word $Z$ such that $X\mathop\sim\limits^{\alpha-1}Z$ and
	\begin{equation}\label{9} f_{\alpha-1}(R\ast E\ast S, X, Z)=R_1\ast C^tC_1\ast S_1,
	\end{equation}
	where one of the words $C^n$ or $C^{-n}$ is a cyclic shift of $A^n_j$. For an arbitrary $d\in \mathcal{D}$ there exists by \eqref{5} an element $d'\in \mathcal{D}$ such that $Xd\mathop\eqsim\limits^{\alpha-1}Z(d'd)$. Then for an arbitrary $d\in \mathcal{D}$ we define a generalized real reversal of rank $\alpha$ of the occurrence $R\ast E\ast Sd$ to be any transition of the form
	\begin{equation}\label{10} RESd\to Y,
	\end{equation}
	where
	\begin{equation}\label{11} Y\mathop\eqsim\limits^{\alpha-1}[R_1,\,[C^{-n+t}d_j^\sigma,\,C_1S_1(d'd)]^\mathcal{D}_{\alpha-1}]^\mathcal{D}_{\alpha-1} \end{equation}
	and $\sigma=1$ or $\sigma=-1$ according to whichever of $C^{-n}$ or $C^n$ is a cyclic shift of the word $A^n_j\in\mathcal{E}_\alpha$. \textit{From the  commutativity of the group $ \mathcal{D} $ implies that this definition is correct}.
	
	We will say that the words $X,Y\in \mathcal{P}_\alpha^\mathcal{D}$ are generalized equivalent in rank $\alpha$ if either $X\mathop\eqsim\limits^{\alpha-1}Y$ or there is s sequence of generalized real reversals of rank $\alpha$ which transform $X$ into $Y$. If in addition $X,Y\in \mathcal{R}_\alpha^\mathcal{D}$, then we will write $X\mathop\eqsim\limits^{\alpha}Y$.
	
	It follows from the relation \eqref{5} of rank $\alpha-1$, from the relations \eqref{16} and from the \eqref{17} and the definition of generalized reversal of rank $\alpha$ that for any $P,Q \in \mathcal{R}_{\alpha-1}$ the transition $P\to Q$ is a real reversal of rank $\alpha$ of the occurrence $R\ast E\ast S$ in the word $P$ if and only if there exists $d\in \mathcal{D}$ such that for an arbitrary element $d'\in D$ the transition $Pd'\to Q(dd')$ is a generalized real reversal of rank $\alpha$ of the occurrence $R\ast E\ast Sd$. It follows easily from this that \eqref{5} holds for rank $\alpha$.
	
	Let us prove the relation \eqref{6} for rank $\alpha$.
	
	Making use of \eqref{5} we define a function $f^\mathcal{D}_\alpha(V;X,Y)$ which maps $\text{Ker}(\alpha, X)$ onto $\text{Ker}(\alpha, Y)$, where  $X\mathop\eqsim\limits^{\alpha}Y$. Namely, if  $Pd\mathop\eqsim\limits^{\alpha}Qd'$, where $P\mathop\sim\limits^{\alpha}Q$, then for an arbitrary kernel $R\ast E\ast Sd\in \text{ß}(\alpha, Pd)$ we set $$f^D_\alpha(R\ast E\ast Sd;\, Pd,\,Qd')\rightleftharpoons f_\alpha(R\ast E\ast S;\, P,\,Q)d.$$ For an arbitrary generalized real reversal $X\to Y$ of rank $\alpha$, where $X,Y\in \mathcal{P}^\mathcal{D}_\alpha$, we also have a mapping $f^\mathcal{D}_\alpha(V;X,Y)$.
	
	The properties of generalized real reversals of rank $\alpha$ are analogues 
	to the properties of the ordinary real reversals of rank $\alpha$.
	
	In order to prove the symmetry of the relation $\mathop\eqsim\limits^{\alpha}$ we need to show that for any active occurrence $R\ast E\ast S$ of rank $\alpha$ in the word $X\in  \mathcal{R}_\alpha$ the above definition of generalized reversal of rank $\alpha$ of the occurrence  $R\ast E\ast Sd$ does not depend on the choice of the intermediate word $Z$ for which the relations $X\mathop\sim\limits^{\alpha-1}Z$ and \eqref{9} hold. Suppose that we have another word $Z_1$ such that $X\mathop\sim\limits^{\alpha-1}Z_1$ and
	\begin{equation}\label{12} f_{\alpha-1}(R\ast E\ast S, X, Z_1)=R_2\ast D^hD_1\ast S_2,
	\end{equation}
	where one of the words $D^n$ or $D^{-n}$ is a cyclic shift of some element $A^n_k$ from the set $\mathcal{E}$. Then, together with the reversal \eqref{10}, we now have a further generalized real reversal of rank $\alpha$ of the occurrence $R\ast E\ast Sd$, namely: $$RESd\to Y_1,$$ where
	\begin{equation}\label{13} Y_1\mathop\eqsim\limits^{\alpha-1}[R_2,\,[D^{-n+t}d_k^{\sigma_1},\,D_1S_2d'']^\mathcal{D}_{\alpha-1}]^\mathcal{D}_{\alpha-1}, \end{equation}
	$Xd\mathop\eqsim\limits^{\alpha-1}Z_1d''$ and ${\sigma}_1=1$ or $\sigma_1=-1$ according to whether $D^{-n}$ or $D^{n}$ is a cyclic shift of the word $A^n_k$. Since in addition $\text{Ðîä}(D^hD_1,\,C^tC_1)$ follows from \eqref{9} and \eqref{12}, we by definition \cite[Ch. VI, Definition 2.1]{A} 	
	have that $A_j\equiv A_k$. Thus, $\sigma_1=\sigma$. It follows from \eqref{9} and \eqref{12} that
	\begin{equation}\label{14}
	f_{\alpha-1}(R_1\ast C^tC_1\ast S_1, Z, Z_1)=R_2\ast D^hD_1\ast S_2. \end{equation}
	
	It follows from item (b) of Definition \cite[Ch. VI, 2.1]{A} that the occurrence $R_1\ast C^tC_1\ast S_1$ is periodised in rank $\alpha-1$. Then by \cite[Ch. II, 7.21]{A} it follows from  \eqref{14} that $C^tC_1\equiv D^hD_1$,  that is, $C\equiv D$, $C_1\equiv D_1$ and $t=h$. By \cite[Ch. IV, 2.35]{A} we have the relations $$R_1\mathop\sim\limits^{\alpha-1}R_2\quad \text{and}\quad S_1\mathop\sim\limits^{\alpha-1}S_2.$$ Then, in view of \eqref{5} there exist elements $a,b\in \mathcal{D}$ such that
	\begin{equation}\label{15}
	R_1\mathop\eqsim\limits^{\alpha-1}R_2a\quad\text{and}\quad S_1\mathop\eqsim\limits^{\alpha-1}S_2b.
	\end{equation}
	
	By the inductive hypothesis coupling in rank $\alpha-1$ has all the analogue properties to the properties of ordinary coupling in rank $\alpha-1$. Therefore, it follows from \eqref{15} that $$R_1C^tC_1S_1d'\mathop\eqsim\limits^{\alpha-1}R_2D^tD_1S_2(d'ab).$$
	
	Then $$Z_1d''\mathop\eqsim\limits^{\alpha-1}Xd\mathop\eqsim\limits^{\alpha-1}Zd'\mathop\eqsim\limits^{\alpha-1}Z_1(d'ab),$$ whence, by relation \eqref{6} for rank $\alpha-1$, it follows that $d''=d'ab.$ Further, using associativity and a single-valuedness of generalized coupling in rank $\alpha-1$ as well as relations \eqref{11}, \eqref{8} and \eqref{13}, we get $$Y\mathop\eqsim\limits^{\alpha-1}[R_1,\,[C^{-n+t}d_j^\sigma,\,C_1S_1d']^\mathcal{D}_{\alpha-1}]^\mathcal{D}_{\alpha-1} \mathop\eqsim\limits^{\alpha-1}[R_2a,\,[D^{-n+t}d_k^{\sigma_1},\,D_1S_2d'b]^\mathcal{D}_{\alpha-1}]^\mathcal{D}_{\alpha-1}\mathop\eqsim\limits^{\alpha-1}$$ $$\mathop\eqsim\limits^{\alpha-1}[R_2,\,[D^{-n+t}d_k^{\sigma_1},\,D_1S_2d'ab]^\mathcal{D}_{\alpha-1}]^\mathcal{D}_{\alpha-1} \mathop\eqsim\limits^{\alpha-1}Y_1.$$
	
	Thus, we have proved that the result of any two generalized real reversals of rank $\alpha$ of a given occurrence $R\ast E\ast Sd$ are generalized equivalent in rank $\alpha-1$. It follows that if $X\to Y$ is a generalized real reversal of a kernel $\text{Ker}(\alpha, X)$, then $Y\to X$ is a generalized real reversal of rank $\alpha$ of the kernel $f^\mathcal{D}_\alpha(V;\,X,\, Y)$. Thus, the relation  $\mathop\eqsim\limits^{\alpha}$ is symmetric. Transitivity follows immediately from the definition.
	
	Further, using analogues of assertions \cite[Ch. IV, 2.1]{A} and \cite[Ch. IV, 2.5]{A} for generalized real reversal of rank $\alpha$ we prove the analogue of \cite[ãë. IV, 2.6]{A}: if $X\mathop\eqsim\limits^{\alpha}Y$, then either  $X\mathop\eqsim\limits^{\alpha-1}Y$, or there exists a simple sequence of generalized real reversals of rank $\alpha$ transforming $X$ to $Y$.
	
	Assume that $Qd\mathop\eqsim\limits^{\alpha}Qd'$, where $d\neq d'$ in $\mathcal{D}$. Since by the induction assumption \eqref{6} is true for rank $\alpha-1$, we have $\neg\, Qd\mathop\eqsim\limits^{\alpha-1}Qd'$. Consequently,  there exist a simple sequence $$Qd\equiv X_1\to X_2\to \cdots\to X_i\to X_{i+1}\to\cdots\to X_t\equiv Qd'$$ of rank $\alpha$. Then by $Y_i$ denoting the result of deleting the letters $d^\pm_j$ from $ X_i $, where $j\in \mathbb{N}$, we get a simple sequence $$Q\equiv Y_1\to Y_2\to \cdots\to Y_i\to Y_{i+1}\to\cdots\to Y_t\equiv Q$$ of ordinary real reversals of rank  $\alpha$. But by \cite[Ch. III, 1.5]{A} and \cite[Ch. IV, 2.4]{A} this sequence can not be simple, since $I_\alpha(Q,\,Q)=0$. Thus relation \eqref{6} is true for rank $\alpha$ as well. 
	
	Based on relations \eqref{5} and \eqref{6} we can prove for $\mathop\eqsim\limits^{\alpha}$ the analogue of each of the assertion that was proved for  $\mathop\sim\limits^{\alpha}$ in \cite[Ch. IV]{A} without significant difficulty.
	
	The operation of generalized coupling in rank $\alpha$ is defined for arbitrary words $Pd,Qd'\in \mathcal{R}^\mathcal{D}_{\alpha}$ by the equations
	\begin{equation}\label{161} [Pd,Qd']^\mathcal{D}_{\alpha}=[P_1d'',Q_1d''']^\mathcal{D}_{0}=[P_1,Q_1]_0(d''d''') \end{equation}
	under the assumption that the following conditions are satisfied:
	\begin{equation}\label{171} Pd\mathop\eqsim\limits^{\alpha}P_1d'', \,\, Qd'\mathop\eqsim\limits^{\alpha}Q_1d'''\,\,\,\text{and}\,\,\,[P_1,Q_1]_0\in\mathcal{R}_\alpha. \end{equation}
	The existence of an operation like this defined everywhere on $\mathcal{R}^\mathcal{D}_{\alpha}$, follows immediately from \eqref{5} and the existence of ordinary coupling in rank $\alpha$. The fact that it is associative and single-valued is proved just as the corresponding properties of ordinary coupling were proved in Chapter V \cite{A}.
	
	\section{Auxiliary group $\Gamma^\mathcal{D}(m,n,\alpha)$}
	
	Using the introduced notions we construct an auxiliary group $\Gamma^\mathcal{D}(m,n,\alpha)$ whose  elements are equivalence classes $\mathcal{R}^\mathcal{D}_{\alpha}$ under $\mathop\eqsim\limits^{\alpha}$, and the group operation is defined in terms of generalized coupling in rank $\alpha$. As in Lemma \cite[Ch. VI, 1.4]{A} it is  easy to check that the set $\Gamma^\mathcal{D}(m,n,\alpha)$ with respect to this operation is a group. Let us describe this group by generators and defining relations. To this end, we denote by $A_{\mathcal{D}}(m,n,\alpha)$ the group with the generators $a_1, a_2..., a_m, d_1, d_2, ..., d_i, ....$ and the system of defining relations of the forms \eqref{k0}, \eqref{k1} and  \eqref{k2} for all those $j\in \mathbb{N}$ for which $A_j\in \mathop\bigcup\limits^\alpha_{t=1}\mathcal{E}_t$.
	
	\begin{lemma}\label{l3} For any words $X,Y\in \mathcal{R}^\mathcal{D}_{\alpha}$ the relation $$X=Y\quad\text{in}\quad A_{\mathcal{D}}(m,n,\alpha)\quad\Leftrightarrow \quad X\mathop\eqsim\limits^{\alpha}Y$$
	\end{lemma}
	holds.
	\begin{proof} The implication  from left to right is proved in exactly the same way as the analogues fact was proved for $X\mathop\sim\limits^{\alpha}Y$ in Lemma \cite[Ch. VI, 2.3]{A}.
		
	Let us prove the implication from right to left. From the definition of generalized real reversal of rank $\alpha$ implies that the transition $$A^{3q}_j\to A^{-n+3q}_jd_j$$ is a generalized real reversal in rank $\le\alpha$ for any word $A_j\in \mathop\bigcup\limits^\alpha_{t=1}\mathcal{E}_t$, that is, $A^{3q}_j\mathop\eqsim\limits^{\alpha} A^{-n+3q}_jd_j$. Consequently, in $\Gamma^\mathcal{D}(m,n,\alpha)$ all relations of the form \eqref{k2} hold in the group $A_{\mathcal{D}}(m,n,\alpha)$. Since all elements of $\mathcal{D}$ belong to the set $\mathcal{R}^\mathcal{D}_\alpha$ (for any $\alpha$) and $r\mathop\eqsim\limits^{0}1$, then $r\mathop\eqsim\limits^{\alpha}1$, where $r\in \mathcal{R}$ (see \eqref{D}). Further, by virtue of \eqref{161} and \eqref{171} we also have  $[d_j,a_k]^\mathcal{D}_\alpha=a_kd_j=[a_k, d_j]^\mathcal{D}_\alpha$. Hence, the relations \eqref{k0} and \eqref{k1} also hold in the group $\Gamma^\mathcal{D}(m,n,\alpha)$. Therefore, the equality $X=Y$ in $A_{\mathcal{D}}(m,n,\alpha)$ implies that the equivalence classes of words $X$ and $Y$ are equal in $\Gamma^\mathcal{D}(m,n,\alpha)$ for any $X,Y\in \mathcal{R}^\mathcal{D}_{\alpha}$, that is, $X\mathop\eqsim\limits^{\alpha}Y.$ \end{proof}
	
	\section{The proof of Theorem \ref{t1}}
	
	By Lemmas \ref{l1} and \ref{l2} the subgroup $\mathcal{Z}$ generated by the elements $\{d_j|j\in \mathbb{N}\}$ coincides with the centre of $A_{\mathcal{D}}(m,n)$, the quotient group $A_{\mathcal{D}}(m,n)/ \mathcal{Z}$ is isomorphic to the free Burnside group $B(m,n)$, the verbal subgroup of the group $A_\mathcal{D}(m,n)$ generated by the word $x^n$ coincides with $\mathcal{Z}$ and in the group $A_\mathcal{D}(m,n)$ the identity relation $[x^n,y]=1$ holds. Therefore, Theorem \ref{t1} will be proved if we show that the abelian group $\mathcal{D}$ with the same generators $\{d_j|j\in \mathbb{N}\}$ is embedded into the group $A_\mathcal{D}(m,n)$ and, hence, coincides with $\mathcal{Z}$.
	
	First, we verify that the group $\mathcal{D}$ is embedded into $\Gamma^\mathcal{D}(m,n,\alpha)$ for any $\alpha$. Suppose that for some elements $d', d''\in \mathcal{D}$ there is an equivalence $d'\mathop\eqsim\limits^{\alpha}d''.$ Then $d'\mathop\eqsim\limits^{\gamma}d''$ for every rank $\gamma\ge\alpha$ because by definition we have $d', d''\in \mathcal{R}_\gamma^\mathcal{D}$. Hence, in view of the relation \eqref{6} we obtain that $d'=d''$ in the abelian group $\mathcal{D}$, that is, $\mathcal{D}$ is embedded into the group $\Gamma^\mathcal{D}(m,n,\gamma)$.
	
	According to Lemma \ref{l3} we obtain that $\mathcal{D}$ is embedded into $A_{\mathcal{D}}(m,n,\gamma)$ for every rank $\gamma\ge\alpha$. Then $\mathcal{D}$ is also embedded into the group $A_{\mathcal{D}}(m,n)$ because the set of defining relations \eqref{k0}--\eqref{k2} of the group $A_{\mathcal{D}}(m,n)$ is the union of sets of defining relations of groups $A_{\mathcal{D}}(m,n,\gamma)$ for all $\gamma\ge\alpha$. Theorem \ref{t1} is proved.

\end{document}